\documentclass[11pt,leqno]{amsart}
\usepackage{amsmath,amssymb,amsthm}
\usepackage{hyperref}
\usepackage[UKenglish]{babel}
\usepackage{microtype}
\usepackage{color}
\usepackage[shortlabels]{enumitem}

\usepackage{xcolor}
\hypersetup{
    colorlinks,
    citecolor={blue},
}

\newtheorem{theorem}{Theorem}[section]
\newtheorem{lemma}[theorem]{Lemma}

\newtheorem{proposition}[theorem]{Proposition}
\newtheorem{corollary}[theorem]{Corollary}

\theoremstyle{definition}
\newtheorem{definition}[theorem]{Definition}
\newtheorem{question}{Question}

\theoremstyle{remark}
\newtheorem{remark}[theorem]{Remark}
\numberwithin{equation}{section}
\newcommand{\abs}[1]{\lvert#1\rvert}
\def\fnote#1{\footnote}

\def\ignora#1{}
\def\n3#1{\left\vert  \! \left\vert \! \left\vert \, #1 \, \right\vert \!
  \right\vert \! \right\vert }

\newcommand{\pten}{\ensuremath{\widehat{\otimes}_\pi}}

\def\<{\langle}
\def\>{\rangle}
\def\vspan{\operatorname{span}}

\def\norma{\|\hspace*{.35cm}\|}
\def\diam{\operatorname{diam}}

\def\Mid{\operatorname{Mid}}

\newcommand{\Natural}{\mathbb N}

\newcommand{\Real}{\mathbb R}

\newcommand{\set}[1]{\left\{#1\right\}}
\newcommand{\restricted}{\mathord{\upharpoonright}}
\newcommand{\dist}{\mathop{\mathrm{dist}}\nolimits}

\newcommand{\norm}[1]{\left\Vert#1\right\Vert}
\newcommand{\duality}[1]{\left\langle#1\right\rangle}
\newcommand{\clco}{\mathop{\overline{\mathrm{conv}}}\nolimits}
\newcommand{\closedball}[1]{B_{#1}}

\newcommand{\sphere}[1]{S_{#1}}

\newcommand{\Free}{{\mathcal F}}

\newcommand{\Lip}{{\mathrm{Lip}}_0}
\newcommand{\justLip}{{\mathrm{Lip}}}

\newcommand{\conv}{\mathop\mathrm{conv}}
\newcommand{\ext}[1]{\mathrm{ext}\left(#1\right)}
\newcommand{\strexp}[1]{\mathrm{str exp}\left(#1\right)}

\begin{document}

\title{ A characterisation of the Daugavet property in spaces of Lipschitz functions }

\author {Luis Garc\'ia-Lirola}
\thanks{The research of L. Garc\'ia-Lirola was supported by the grants MINECO/FEDER MTM2014-57838-C2-1-P and Fundaci\'on S\'eneca CARM
19368/PI/14}
\address[L. Garc\'ia-Lirola]{Universidad de Murcia, Facultad de Matem\'aticas, Departamento de Matem\'aticas, 
30100 Espinardo (Murcia), Spain}
\email{luiscarlos.garcia@um.es}

\author{Anton\'in Proch\'azka}
\address[A. Proch\'azka]{Universit\'e Bourgogne Franche-Comt\'e, Laboratoire de Math\'ematiques UMR 6623, 16 route de Gray,
25030 Besan\c con Cedex, France}
\email{antonin.prochazka@univ-fcomte.fr}

\author{ Abraham Rueda Zoca }\thanks{The research of A. Rueda Zoca was supported by a research grant Contratos predoctorales FPU del Plan Propio del Vicerrectorado de Investigaci\'on y Transferencia de la Universidad de Granada, by MINECO (Spain) Grant MTM2015-65020-P and by Junta de Andaluc\'ia Grants FQM-0185.}
\address[A. Rueda Zoca]{Universidad de Granada, Facultad de Ciencias.
Departamento de An\'{a}lisis Matem\'{a}tico, 18071-Granada
(Spain)} \email{ abrahamrueda@ugr.es}
\urladdr{\url{https://arzenglish.wordpress.com}}



\begin{abstract}
We study the Daugavet property in the space of Lipschitz functions $\Lip(M)$ for a complete metric space $M$.
%We characterise which spaces of Lipschitz functions $\Lip(M)$ enjoy the Daugavet property in terms of a geometric condition on the underlying metric space $M$. 
Namely we show that $\Lip(M)$ has the Daugavet property if and only if $M$ is a length space. 
This condition also characterises the Daugavet property in the Lipschitz free space $\mathcal{F}(M)$.
Moreover, when $M$ is compact, we show that either $\mathcal{F}(M)$ has the Daugavet property or its unit ball has a strongly exposed point. 
If  $M$ is an infinite compact subset of a strictly convex Banach space then the Daugavet property of $\Lip(M)$ is equivalent to the convexity of $M$.
%%Tony 01/08: El resultado anunciado en la ultima frase no me parece el mas importante para anunciar en el abstracto. Como maximo para  motivos comerciales: para venderselo a los que conocen solamente los convexos...
%% LC 03/08 Ya, pero en el abstract de IKW ellos destacan algo parecido. Lo que veais. 

\end{abstract}

\maketitle \markboth{L. GARC\'IA-LIROLA, A. PROCH\'AZKA AND A. RUEDA ZOCA}{A CHARACTERISATION OF THE DAUGAVET PROPERTY ...}

\section{Introduction}

A Banach space $X$ is said to have the Daugavet property if every rank-one operator $T:X\longrightarrow X$ satisfies the equality
\begin{equation}\label{ecuadauga}
\Vert T+I\Vert=1+\Vert T\Vert,
\end{equation}
where $I$ denotes the identity operator. The previous equality is known as \emph{Daugavet equation} because I.~Daugavet proved in \cite{dau} that every compact operator on $\mathcal C([0,1])$ satisfies (\ref{ecuadauga}). Since then, many examples of Banach spaces enjoying the Daugavet property have appeared. E.g. $\mathcal C(K)$ for a perfect compact Hausdorff space $K$; $L_1(\mu)$ and $L_\infty(\mu)$ for a non-atomic measure $\mu$; or preduals of Banach spaces with the Daugavet property (see \cite{kkw,kssw,wer} and references therein for a detailed treatment of the Daugavet property).

%% Abraham 12/04: Creo que aquí va un poco forzado que la Daugavet pasa a preduales. Creo que se podría poner en los ejemplos, algo como "measure "\mu" or preduals of Banach spaces with the Daugavet property".
%% LC 12/04. Como veas. En tal caso cambiaría el "or" después de $K$ por una coma.
%% Tony 13/05: De momento lo dejaria donde esta.
%% Tony 01/08: Algunos retoques hechos, a ver si parece mejor ahora.
In \cite[Section 6]{wer} it is asked whether the space $\Lip([0,1]^2)$ of Lipschitz functions over the unit square enjoys or not the Daugavet property. A positive answer was given in \cite{ikw}, where it was shown, among other results, that $\Lip(M)$ has the Daugavet property whenever $M$ is a length metric space.

Here we prove the converse implication, thus obtaining our main theorem (Theorem~\ref{caracolocal}) which completely characterises those complete metric spaces $M$ such that $\Lip(M)$ has the Daugavet property. 
As a consequence of Theorem~\ref{caracolocal} we also get that the space $\Lip(M)$ has the Daugavet property if, and only if, its canonical predual $\Free(M)$ (see the formal definition below) has the Daugavet property, extending the corresponding result in the compact case which was proved in \cite{ikw}.

This paper is organised as follows. 
In Section \ref{sec:metric} we introduce necessary definitions and establish several results concerning length and geodesic metric spaces, in particular we show that a complete local space is a length space. We also study sufficient conditions for a metric space to be geodesic. 
Section \ref{sec:daugavet} is devoted to the proof of the main theorem, the charaterisation of Lipschitz free spaces and spaces of Lipschitz functions with the Daugavet property. 
Section \ref{sec:exposed} includes a characterisation of strongly exposed points in $\closedball{\Free(M)}$ (Theorem~\ref{th:charstrexp}). 
We use this result to prove in Corollary \ref{caracompa} that, when $M$ is compact, the Daugavet property of $\Free(M)$ is equivalent to the absence of strongly exposed points of $\closedball{\Free(M)}$. 
%% LC 13/07 actualizo la frase. Se queda un poco largo, si queréis podeis resumirlo. 
It is not clear whether the absence of strongly exposed points of $\closedball{\Free(M)}$ implies in general that $M$ is a length space. 
In the first part of Section~\ref{sec:remarks} we gather some partial evidence to support such a conjecture. 
In the second part of Section~\ref{sec:remarks} we study the Daugavet property in the spaces of vector-valued functions $\Lip(M,X)$. 
This is used to give new examples of spaces of linear bounded operators and of projective tensor products enjoying the Daugavet property.

\textbf{Notation:} Throughout the paper we will only consider real Banach spaces. Given a Banach space $X$, we will denote the closed unit ball and the unit sphere of $X$ by $B_X$ and $S_X$ respectively. We will also denote by $X^*$ the topological dual of $X$. 

By a \textit{slice} of the unit ball $B_X$ of a Banach space $X$ we will mean a set of the following form
\[ S(B_X,f,\alpha):=\{x\in B_X:f(x)>1-\varepsilon\}\]
where $f\in S_{X^*}$ and $\alpha>0$. Notice that slices are non-empty relatively weakly open and convex subsets of $B_X$ whose complement is also convex. 

Given a metric space $M$ and a point $x\in M$, we will denote by $B(x,r)$ the closed unit ball centered at $x$ with radius $r$. 
%% Tony 15/05: realmente la usamos abierta? Me gustaria hacerla cerrada, si posible. O por lo menos asegurarnos que todas las veces que usamos esta notacion (19 veces), es la misma.
%% LC 03/08 No he encontrado ninguna razón para ponerla abierta, así que la pongo cerrada
Let $M$ be a metric space with a distinguished point $0 \in M$.
The couple $(M,0)$ is commonly called a \emph{pointed metric space}.
By an abuse of language we will say only ``let $M$ be a pointed metric space'' and similar sentences.
The vector space of Lipschitz functions from $M$ to $\Real$ will be denoted by $\justLip(M)$.
Given a Lipschitz function $f\in \justLip(M)$, we denote its Lipschitz constant by
\[ \norm{f}_{L} = \sup\set{ \frac{|f(x)-f(y)|}{d(x,y)} : x,y\in M, x\neq y}. \]
This is a seminorm on $\justLip(M)$ which is clearly a Banach space norm on the space $\Lip(M)\subset \justLip(M)$ of Lipschitz functions on $M$ vanishing at $0$. 
It is well-known that $\Lip(M)$ is a dual Banach space, whose canonical predual is the \emph{Lipschitz free space} 
\[ \mathcal F(M) : = \overline{\vspan}\{\delta_x : x\in M\}\subset \Lip(M)^* \]
where $\delta_x(f):= f(x)$ for every $x\in M$ and $f\in \justLip(M)$ (see~\cite{GK,wea}, or~\cite{cdw} for the most elementary proof of this fact).
If $N$ is a dense subset of $M$ then $\mathcal F(N)$ and $\mathcal F(M)$ are isometrically isomorphic Banach spaces as every Lipschitz function on $N$ extends uniquely to a Lipschitz function on $M$ with the same Lipschitz constant. 
Thus the results about $\Free(M)$ or $\Lip(M)$ can be stated for \emph{complete} $M$ without any loss of generality.

We finally recall two geometric characterisations of the Daugavet property in terms of the slices of the unit ball. We refer the reader to \cite{kssw,wer} for a detailed proof.

\begin{theorem}\label{caragendauga}
Let $X$ be a Banach space. The following assertions are equivalent:
\begin{enumerate}
\item\label{caragendauga1} $X$ has the Daugavet property.
\item\label{caragendauga2} For every $x\in S_X$, every slice $S$ of $B_X$ and every $\varepsilon>0$ there exists another slice $T$ of the unit ball such that $T\subseteq S$ and such that
$$\Vert x+y\Vert>2-\varepsilon$$
holds for every $y\in T$.
\item\label{caragendauga3} For every $x\in S_X$ and every $\varepsilon>0$ the following equality holds:
$$B_X=\clco(\{y\in (1+\varepsilon) B_X: \Vert y-x\Vert>2-\varepsilon\} ).$$
\end{enumerate}
\end{theorem}

Note that (\ref{caragendauga3}) is particularly useful in those Banach spaces in which there is not a complete description of the dual space.

\section{Length spaces and geodesic spaces}\label{sec:metric}

\begin{definition}
 We will say that a metric space $(M,d)$ is a \emph{length space} if, for every pair of points $x,y \in M$, the distance $d(x,y)$ is equal to the infimum of the length of rectifiable curves joining them. Moreover, if that infimum is always attained then we will say that $M$ is a \emph{geodesic space}. 
\end{definition}

These definitions are standard, for more details see e.g.~\cite{BH}. Geodesic spaces and length spaces were considered in~\cite{ikw}, where they are called metrically convex spaces and almost metrically convex spaces, respectively. 

The following lemma is well-known and easy to prove, see~\cite{BBI}.

\begin{lemma}\label{lemma:charlength}
 Let $(M,d)$ be a complete metric space. Then
\begin{enumerate}[(a)]
\item $M$ is a geodesic space if and only if for every $x,y\in M$ there is $z\in M$ such that $d(x,z)=d(y,z)=\frac{1}{2}d(x,y)$. 
\item $M$ is a length space if and only if 
 for every $x,y \in M$ and for every $\delta>0$ the set 
\[
\Mid(x,y,\delta):=B\left (x,\frac{1+\delta}{2}d(x,y)\right ) \cap B\left (y,\frac{1+\delta}{2}d(x,y)\right )  
\]
is non-empty.
\end{enumerate}
\end{lemma}

The next definition comes from~\cite{ikw}.

\begin{definition}
A metric space $M$ is said to be \emph{local} if, for every $\varepsilon>0$ and every Lipschitz function $f\colon M\to\mathbb R$ there exist $u\neq v\in M$ such that $d(u,v)<\varepsilon$ and $\frac{f(u)-f(v)}{d(u,v)}>\Vert f\Vert_L-\varepsilon$. 

Moreover, $M$ is said to be \emph{spreadingly local} if for every $\varepsilon>0$ and every Lipschitz function $f\colon M\to \mathbb R$ the set
\[ \left\{x\in M : \inf_{\delta>0} \norm{f\restricted_{B(x,\delta)}}_{L}>\norm{f}_{L}-\varepsilon\right\}\]
is infinite. 
\end{definition}

It has been proved in~\cite{ikw} that length spaces are spreadingly local and that locality implies spreading locality under compactness assumptions.
But in fact we have the equivalence of the three concepts in general.
\begin{proposition}\label{prop:charlengthspace} 
Let $M$ be a complete 
metric space. The following are equivalent:
\begin{enumerate}[(i)] 
\item\label{it:len1} $M$ is a length space.
\item\label{it:len2} $M$ is spreadingly local.
\item\label{it:len3} $M$ is local.
\end{enumerate}
\end{proposition}
\begin{proof}
\ref{it:len2}$\Rightarrow$\ref{it:len3} is trivial and 
\ref{it:len1}$\Rightarrow$\ref{it:len2} was proved in~\cite{ikw}, see the remark after Proposition~2.3.
For the reader's convenience we sketch the main idea. 
For a given $f \in \closedball{\Lip(M)}$ and $\varepsilon>0$ let $x,y \in M$ be such that $f(x)-f(y)\geq (1-\frac{\varepsilon^2}{4}) d(x,y)$.
Let $\varphi:[0,d(x,y)(1+\frac{\varepsilon}2)] \to M$ be a 1-Lipschitz map such that $\varphi(0)=x$ and $\varphi(d(x,y)(1+\frac\varepsilon2)=y$.
Then $f(y)=f(x)+\int_0^{d(x,y)(1+\varepsilon/2)} (f\circ \varphi)'(t)\, dt$ and the integrand has to be larger than $1-\frac\varepsilon2$ in a non-negligible subset $A$ of $[0,(1+\varepsilon/2)d(x,y)]$.
It is immediate to check $\varphi(A)$ satisfies the definition of spreading locality for $\varepsilon$.

%Moreover,~\ref{it:len2}$\Rightarrow$\ref{it:len3} is trivial. 
To show that \ref{it:len3}$\Rightarrow$\ref{it:len1}, assume that $M$ is not a length space. 
Then there exist $x,y \in M$ and $\delta>0$ such that $\Mid(x,y,2\delta)=\emptyset$. 
Let us denote $r:=\frac{d(x,y)}2$.
Notice by passing that \[\dist(B(x,(1+\delta)r),B(y,(1+\delta)r))\geq \delta r.\]
Let $f_i\colon M \to \Real$ be defined by 
\[
f_1(t)=\max\set{r-\frac1{1+\delta}d(x,t),0} \mbox{ and }
f_2(t)=\min\set{-r+\frac1{1+\delta}d(y,t),0}.
\]
Clearly $\norm{f_i}_{L}\leq \frac1{1+\delta}$ so $f=f_1+f_2$ is a Lipschitz function. 
Since $f(x)-f(y)=d(x,y)$ we have that $\norm{f}_{L}\geq 1$. 
Moreover we have that $\set{z:f_1(z)\neq 0} \subset B(x,(1+\delta)r)$ and 
$\set{z:f_2(z)\neq 0} \subset B(y,(1+\delta)r)$.
It follows that if $\frac{f(u)-f(v)}{d(u,v)}>\frac{1}{1+\delta}$ then $u \in B(x,(1+\delta)r)$ and $v \in B(y,(1+\delta)r)$. But then $d(u,v) \geq \delta r$ and so $M$ is not local. This shows that~\ref{it:len3}$\Rightarrow$\ref{it:len1}.
\end{proof}

It is clear from Lemma~\ref{lemma:charlength} that every compact length space is geodesic. 
But the compactness is not always needed for this implication to hold. 
Indeed, in some particular cases, being a length space automatically implies being a geodesic space. 
For instance, this is the case for weak*-closed length subsets of dual Banach spaces. In what follows we wish to study geometric properties of a Banach space $X$ that ensure that  every complete length subset is geodesic. 
Let us recall that the Kuratowski index of non-compactness of a set $D\subset X$ is given by
\[ \alpha(D) = \inf\left \{\varepsilon>0 : \exists x_1,\dotsc, x_n \in X, D\subset \bigcup_{i=1}^n B(x_i, \varepsilon)\right \}. \]

\begin{proposition}\label{prop:lenghtvsgeodesic} Assume that $\lim_{\delta\to 0}\alpha(\Mid(x,-x,\delta))=0$ for every $x\in S_X$. Let $M$ be a complete subset of $X$. Then if $M$ is a length space, it is a geodesic space. 
\end{proposition}

\begin{proof}
Let $x,y \in M$ be given, by scaling and shifting we may assume that $x\in S_X$ and $y=-x$. Using Lemma~\ref{lemma:charlength} there is, for every $n\in \Natural$, a point $x_n \in \Mid(x,y,\frac1n)$.
It follows by our hypothesis and by $\Mid(x,y,\frac1{n+1})\subset \Mid(x,y,\frac1n)$ that $\lim_{n\to \infty}\alpha(\{x_k:k\geq n\})=0$.
Therefore for every $\varepsilon>0$ there is $N>0$ such that $\{x_n:n\geq N\}$ can be covered by finitely many balls of radius $\varepsilon$.
This suffices for selecting a Cauchy subsequence. Since $M$ is complete, we have that its limit $z$ belongs to $M$.
It is now clear that $d(x,z)\leq 1$ and $d(y,z)\leq 1$ hence $z$ is a metric midpoint between $x$ and $y$. Now Lemma~\ref{lemma:charlength} gives that $M$ is geodesic.
\end{proof}

The hypothesis of Proposition~\ref{prop:lenghtvsgeodesic} admits the following reformulation in terms of an asymptotic property of the Banach space $X$.

\begin{proposition} Let $x\in S_X$. The following are equivalent:
\begin{enumerate}[(i)]
\item\label{caraamuc1} $\lim_{\delta\to 0} \alpha(\Mid(x,-x,\delta))=0$.
\item\label{caraamuc2} For every $0<t<1$ there is $\delta>0$ and a finite codimensional subspace $Y\subset X$ such that 
\[ \inf_{y\in S_Y} \max\{||x+ty||,||x-ty||\}> 1+\delta \]
%Tony 04/08: Cambio la desigualdad por una estricta para que la condicion describa exactamente que $ty \notin \Mid(x,-x,\delta)$.
\end{enumerate}
\end{proposition}
\begin{proof}
%Tony 25/04: Deberiamos hacer por lo menos un sketch... (en la version final).
Follow the same arguments as in~\cite[Theorem 2.1]{DKRRZ16}. 
Let us sketch the main idea for reader's convenience.
If \ref{caraamuc2} fails, then for some $t>0$ and every $\delta>0$ it is easy to construct inductively a $t$-separated sequence in $\Mid(x,-x,\delta)$ showing that $\alpha(\Mid(x,-x,\delta))\geq t/2$.

Conversely, let $t>0$ be given and let $Y$ and $\delta>0$ be as in \ref{caraamuc2}. 
Since $\Mid(x,-x,\delta)$ is a ball of an equivalent norm on $X$, Lemma~2.13 of \cite{JLPS} shows that there is a finite dimensional $Z \subset X$ so that 
\[
\Mid(x,-x,\delta) \subset (Z \cap \Mid(x,-x,\delta)) + 3(Y \cap \Mid(x,-x,\delta)).
\]
Since we have for every $y \in \sphere{Y}$ that $ty \notin \Mid(x,-x,\delta)$, it follows by convexity that
$Y \cap \Mid(x,-x,\delta)\subset t\closedball{X}$.
Therefore $\alpha(\Mid(x,-x,\delta))\leq 6t$.
\end{proof}

In~\cite{DKRRZ16} the \emph{asymptotic midpoint uniformly convex} spaces (AMUC, for short) were introduced as those Banach spaces in which $\lim_{\delta\to 0}\alpha(\Mid(x,-x,\delta))=0$ uniformly in $x\in S_X$, or, in other words, the same $\delta>0$ works for all $x\in S_X$ in the condition \ref{caraamuc2} above. I.e. for every $0<t<1$ there is $\delta>0$ such that 
\[ \inf_{x\in S_X}\sup_{\dim X/Y <\infty}\inf_{y\in S_Y} \max\{||x+ty||,||x-ty||\}\geq 1+\delta. \]
In particular, every AUC space is also AMUC.

It is clear that if 
\begin{equation}\label{e:MLUR}
\lim_{\delta\to 0}\diam(\Mid(x,-x,\delta))=0 \mbox{ for every } x \in S_X
\end{equation}
then the hypothesis of Proposition~\ref{prop:lenghtvsgeodesic} is satisfied. 
The norms which satisfy \eqref{e:MLUR} are called \emph{midpoint locally uniformly rotund} (MLUR). 
For example, one can easily see that LUR norms are MLUR (see~\cite[Proposition 5.3.27]{Megginson}).

We are going to resume these comments into the following corollary.

\begin{corollary} 
A complete length subset $M$ of a Banach space $X$ is geodesic if any of the following conditions is satisfied:
\begin{itemize}
\item[a)] $X=Y^*$ for some Banach space $Y$ and $M$ is w$^*$-closed (in particular if $M$ is a compact)
\item[b)] $X$ is AMUC (in particular if $X$ is AUC, for example $X=\ell_p$, $1\leq p <\infty$)
\item[c)] $X$ is MLUR (in particular if $X$ is LUR).
\end{itemize}
\end{corollary}

To conclude this section we are going to discuss another metric notion, the property (Z), which is (formally) weaker than being a length space. It was introduced in~\cite{ikw} in order to characterise metrically the local metric spaces in the compact case. We will show in Section~\ref{sec:exposed} that property (Z) characterises the absence of strongly exposed points in $B_{\mathcal F(M)}$. 

\begin{definition}\label{d:propertyZ} A metric space $M$ has \emph{property (Z)} if, for every $x,y\in M$ and $\varepsilon>0$, there is $z\in M\setminus\{x,y\}$ satisfying 
\[ d(x,z)+d(z,y)\leq d(x,y) + \varepsilon\min\{d(x,z),d(z,y)\} \]
\end{definition}

It is proved in~\cite{ikw} that every local metric space has property (Z), and that the converse statement holds in the compact case. Note that the former also follows immediately from Proposition~\ref{prop:charlengthspace} and Lemma~\ref{lemma:charlength}.

Moreover, it is also shown in~\cite{ikw} that every compact subset of a smooth LUR Banach space with property (Z) is convex. As a consequence of Proposition~\ref{prop:charlengthspace} we have the following:

\begin{corollary} Let $M$ a compact metric space with property (Z). Then $M$ is a geodesic space. If moreover $M$ is a subset of a rotund Banach space then $M$ is convex.
\end{corollary}

\begin{proof} It has been proved in~\cite[Proposition 2.8]{ikw} that a compact metric space with property (Z) is local. Thus the first statement above follows from Proposition~\ref{prop:charlengthspace} and the fact that every compact length space is geodesic. Finally, it is easy to show that every geodesic subset of a rotund Banach space is convex.
\end{proof}

Lemma~\ref{lemma:charlength} says that the complete geodesic spaces are those for which every pair of points has a metric midpoint. However, such characterisation can still be weakened by using the concept of \emph{metric segment}. Given a metric space $M$ and a pair of points $x,y\in M, x\neq y$, we consider the \textit{metric segment joining $x$ and $y$} as the following set:
\[ [x,y]:=\{z\in M: d(x,z)+d(z,y)=d(x,y)\}.\]

 \begin{proposition}\label{p:GeneralStandard}
  Let $M$ be a complete metric space. Then $M$ is geodesic if, and only if, for each couple $x\neq y \in M$ there is $z\in [x,y]\setminus \set{x,y}$.
 \end{proposition}

\begin{proof}
Let $x\neq y \in M$ and assume, with no loss of generality, that $d(x,y)=1$. We show that there is an isometry $\phi:[0,1] \to M$ such that $\phi(0)=x$ and $\phi(1)=y$.
We will do this by Zorn lemma. To this end we consider the set $\mathcal A$ of all $(A,\psi)$ where $\set{0,1} \subset A \subset [0,1]$ is closed and $\psi:A \to X$ is an isometry such that $\psi(0)=x$, $\psi(1)=y$, together with the following partial order ``$\leq$'' on $\mathcal A$: $(A,\psi)\leq (B,\xi)$ if $A\subset B$ and $\xi\restricted_A=\psi$.
Now every chain $(A_i,\psi_i)_{i\in I}$ admits an upper bound. Indeed, take $A=\overline{\bigcup_{i\in I}A_i}$ and $\psi(x):=\psi_i(x)$ if $i \in A_i$. This is an isometry on $\bigcup_{i \in I}A_i$, therefore, since $M$ is complete, it extends uniquely to an isometry on the closure.
Now, let $(A,\phi) \in \max \mathcal A$. If $A\neq [0,1]$ then there are $a<b$ such that $a,b \in A$ and $]a,b[ \cap A=\emptyset$.
By the hypothesis there exists $z \in M$ such that $d(\phi(a),z)+d(\phi(b),z)=d(\phi(a),\phi(b))$. We can define $\phi(a+d(\phi(a),z)):=z$ which is easily seen to be an isometry contradicting the maximality of $(A,\phi)$.
\end{proof}

\section{Metric characterisation of the Daugavet property in Lipschitz-free Banach spaces}\label{sec:daugavet}
We start with an auxiliary result, inspired by~\cite[Theorem 3.1]{pr}.

\begin{proposition}\label{t:circular}
Let $M$ be a pointed metric space. The following assertions are equivalent:
\begin{enumerate}[(i)]
\item\label{hola1}  $\mathcal F(M)$ has the Daugavet property.
\item\label{hola2} For each $\varepsilon>0$, each finite subset $N \subset M$ and each norm-one Lipschitz function $g\colon M\longrightarrow \mathbb R$ there are points $u,v \in M$, $u \neq v$, such that $\frac{g(u)-g(v)}{d(u,v)}>1-\varepsilon$ and that 
every $1$-Lipschitz function $f\colon N \to \mathbb R$ admits an extension $\tilde{f}:M \to \mathbb R$ which is $(1+\varepsilon)$-Lipschitz and satisfies $\tilde{f}(u)-\tilde{f}(v)\geq d(u,v)$. 
\item\label{hola3} For each finite subset $N\subseteq M$ and $\varepsilon>0$, there exist $u,v\in M, u\neq v$, such that
\begin{equation}\label{carametridauga}(1-\varepsilon)(d(x,y)+d(u,v))\leq d(x,u)+d(y,v)
\end{equation}
holds for all $x,y \in N$. Moreover, if we define $A:=\{(u,v)\in M^2\setminus \Delta: \mbox{(\ref{carametridauga}}) \mbox{ holds} \}$, where $\Delta:=\{(x,x)\in M^2: x\in M\}$, then 
\[\left\{\frac{\delta_u-\delta_v}{d(u,v)}: (u,v)\in A\right\}\]
is norming for $\Lip(M)$.
\end{enumerate}
\end{proposition}

For the proof of the Proposition~\ref{t:circular} we will need the following lemma.

\begin{lemma}\label{lemanormante}
Let $X$ be a Banach space with the Daugavet property and let $V\subseteq S_X$ be a norming subset for $X^*$. Then, given $x_1,\ldots, x_n\in S_X$, $\varepsilon>0$ and a slice $S$ of $B_X$, there exists $v\in V\cap S$ such that
\[\Vert x_i+ v\Vert>2-\varepsilon\]
holds for every $i\in\{1,\ldots, n\}$.
\end{lemma}

\begin{proof}
Since $X$ has the Daugavet property then, using $n$-times Theorem~\ref{caragendauga}, we can find a slice $T\subseteq S$ of $B_X$ such that for every $y\in T$ one has
\[\Vert x_i+y\Vert>2-\varepsilon\]
for every $i\in\{1,\ldots, n\}$. Since $V$ is norming for $X^*$ it follows from an easy application of Hahn-Banach theorem that $\overline{\conv}(V)=B_X$. Thus $\overline{\conv}(V)\cap T\neq \emptyset$ and so $V\cap T\neq \emptyset$, which concludes the proof.
\end{proof} 

\begin{proof}[Proof of Proposition~\ref{t:circular}]
\ref{hola2}$\Rightarrow$\ref{hola1}: 
Let $\mu \in S_{\Free(M)}$, $g\in \sphere{\Lip(M)}$ and $\varepsilon>0$. We suppose as we may that $N=supp(\mu) \cup \set{0}$ is finite.
By~\ref{hola2} we can find $u,v\in M, u\neq v$ such that $\frac{g(u)-g(v)}{d(u,v)}>1-\varepsilon$. 
Moreover if $f \in \closedball{\Lip(M)}$ is such that $\duality{f,\mu}=\norm{\mu}$  there exists $\tilde{f}\in \Lip(M)$ such that $f=\tilde{f}$ on $N$, $\tilde{f}(u)-\tilde{f}(v)\geq d(u,v)$ and $\norm{\tilde{f}}_L\leq 1+\varepsilon$. 
Now
\[
\left\Vert\frac{\delta_u-\delta_v}{d(u,v)}+\mu \right\Vert\geq \frac{\frac{\tilde{f}(u)-\tilde{f}(v)}{d(u,v)}+\tilde{f}(\mu)}{1+\varepsilon}\geq\frac{1+\Vert \mu\Vert}{1+\varepsilon}.
\]
%Tony 25/04: He cambiado orden de lo de arriba para que corresponda al orden de lo de abajo.
It follows that $\norm{Id+g\otimes\mu}\geq\norm{\frac{\delta_u-\delta_v}{d(u,v)}+\duality{g,\frac{\delta_u-\delta_v}{d(u,v)}}\mu}\geq 2-3\varepsilon$
so we conclude that $\mathcal F(M)$ has the Daugavet property, as desired.\\
\iffalse
Pick finitely-supported measures $\mu_1,\ldots, \mu_n\in S_{\mathcal F(M)}$, $\varepsilon>0$ and a slice $S(B_{\mathcal F(M)},g,\alpha)$ for suitable $g\in \mathcal F(M)^*=\Lip(M)$ and $\alpha>0$. 
Define $N:=\{0\} \cup \bigcup\limits_{i=1}^n supp(\mu_i)$, which is a finite subset of $M$.
For each $i\in\{1,\ldots, n\}$ we can find $g_i\in S_{\Lip(N)}$ such that $g_i(\mu_i)=\Vert \mu_i\Vert$. 
By~\ref{hola2} we can find $u,v\in M, u\neq v$ such that $\frac{g(u)-g(v)}{d(u,v)}>1-\alpha$ and that, for each $i\in\{1,\ldots, n\}$, there exists $f_i\in \Lip(M)$ such that $f_i=g_i$ on $N$, $f_i(u)-f_i(v)\geq d(u,v)$ and $\Vert f_i\Vert\leq 1+\varepsilon$. 
Pick $i\in\{1,\ldots, n\}$. 
Now
\[
\left\Vert\mu_i+\frac{\delta_u-\delta_v}{d(u,v)} \right\Vert\geq \frac{f_i(\mu_i)+\frac{f_i(u)-f_i(v)}{d(u,v)}}{1+\varepsilon}>\frac{g_i(\mu_i)+1}{1+\varepsilon}=\frac{\Vert \mu_i\Vert+1}{1+\varepsilon}.
\]
Moreover, the condition $\frac{g(u)-g(v)}{d(u,v)}>1-\alpha$ implies that $\frac{\delta_u-\delta_v}{d(u,v)}\in S$, so we conclude that $\mathcal F(M)$ has the Daugavet property, as desired.\\
\fi
\ref{hola1}$\Rightarrow$\ref{hola3}: Let $N\subseteq M$ be finite and $\varepsilon>0$. Since $\mathcal F(M)$ has the Daugavet property we can find, using Proposition~\ref{lemanormante},  for every $g\in \sphere{\Lip(M)}$ and every $\alpha>0$ two elements $u\neq v\in M$ such that $\frac{\delta_u-\delta_v}{d(u,v)}\in S(B_{\mathcal F(M)},g,\alpha)$ and that 
\[
\left\Vert \frac{\delta_x-\delta_y}{d(x,y)}+ \frac{\delta_u-\delta_v}{d(u,v)} \right\Vert>2-\varepsilon,
\]
holds for every $x\neq y\in N$.  
By an easy convexity argument (see the proof of~\cite[Theorem 3.1]{pr} for details) we conclude that

\[(1-\varepsilon)(d(x,y)+d(u,v))<d(x,v)+d(u,y)\]
holds for every $x\neq y\in N$. In addition, since $g\in S_{\Lip(M)}$ and $\alpha>0$ were arbitrary we conclude that the set \[\left\{\frac{\delta_u-\delta_v}{d(u,v)}: (u,v)\in A\right\}\]
is norming for $\Lip(M)$, as desired.\\
\ref{hola3}$\Rightarrow$\ref{hola2}: Let $N \subset M$ finite, $g\in S_{\Lip(M)}$ and $\varepsilon>0$ be given. 
By the assumptions, there are $u,v \in M$, $u\neq v$, such that $\frac{g(u)-g(v)}{d(u,v)}>1-\varepsilon$ and  
\[
\frac{1}{1+\varepsilon}(d(x,y)+d(u,v))\leq d(x,u)+d(y,v)
\]
for all $x,y \in N$.
Given a $1$-Lipschitz function $f$ on $N$ we define $\displaystyle\tilde{f}(u)=\inf_{x\in N} f(x)+(1+\varepsilon)d(x,u)$, $\displaystyle\tilde{f}(v)=\sup_{x\in N\cup\{u\}} \displaystyle\tilde{f}(x)-(1+\varepsilon)d(x,v)$. 
Clearly $\tilde{f}$ is $(1+\varepsilon)$-Lipschitz on $N \cup \{u,v\}$ so it admits an $(1+\varepsilon)$-Lipschitz extension to the whole of $M$.
It can be easily seen that $\tilde{f}(u)-\tilde{f}(v)\geq d(u,v)$ (see the proof of~\cite[Theorem 3.1]{pr} for details) so the proof is finished.\end{proof}

The main result of the present article is the following theorem. It improves~\cite[Theorem 3.3]{ikw} where the equivalence between points ii) and iii) is proved for $M$ compact.

\begin{theorem}\label{caracolocal} Let $M$ be a complete pointed metric space. The following assertions are equivalent:
\begin{enumerate}[(i)] 
\item\label{caracolocal1} $M$ is a length space.
\item\label{caracolocal2} $\Lip(M)$ has the Daugavet property. 
\item\label{caracolocal3} $\mathcal F(M)$ has the Daugavet property. 
\end{enumerate}
\end{theorem}

In order to prove Theorem~\ref{caracolocal} we will consider for every $x,y\in M$, $x\neq y$, the function
\[f_{xy}(t):= \frac{d(x,y)}{2}\frac{d(t,y)-d(t,x)}{d(t,y)+d(t,x)}.\]
The properties collected in the next lemma have been proved already in~\cite{ikw2}. They make of $f_{xy}$ a useful tool for studying the geometry of $B_{\mathcal F(M)}$.

\begin{lemma}\label{lemma:IKWfunction} Let $x,y\in M$ with $x\neq y$. We have
\begin{enumerate}[(a)]
\item $\frac{f_{xy}(u)-f_{xy}(v)}{d(u,v)} \leq \frac{d(x,y)}{\max\{d(x,u)+d(u,y),d(x,v)+d(v,y)\}}$ for all $u\neq v \in M$. 
\item $f_{xy}$ is Lipschitz and $\norm{f_{xy}}_{L}\leq 1$. 
\item Let $u\neq v \in M$ and $\varepsilon>0$ be such that $\frac{f_{xy}(u)-f_{xy}(v)}{d(u,v)}>1-\varepsilon$. Then 
\[(1-\varepsilon)\max\{d(x,v)+d(y,v),d(x,u)+d(y,u)\}< d(x,y).\]
\item If $u\neq v \in M$ and $\frac{f_{xy}(u)-f_{xy}(v)}{d(u,v)}=1$, then $u,v\in [x,y]$.
\end{enumerate} 
\end{lemma}

\begin{proof}
Statement (a) follows from the next easily proved fact (see~\cite{ikw2}): Given $u_1,v_1,u_2,v_2>0$, we have
\begin{equation*}
\left| \frac{u_1-v_1}{u_1+v_1} - \frac{u_2-v_2}{u_2+v_2}\right| \leq 2\frac{\max\{|u_1-u_2|,|v_1-v_2|\}}{\max\{u_1+v_1,u_2+v_2\}}.
\end{equation*}
Finally, the statements (b),(c) (resp.\ (d)) are a straightforward consequence of~(a) (resp.\ (c)).
\end{proof}

We will need one more lemma, which is an extension of Lemma 3.2 in~\cite{ikw}.

\begin{lemma}\label{lemalocal} Assume that $\mathcal F(M)$ has the Daugavet property. Then for every $x,y\in M$ and every function $f\in S_{\Lip(M)}$ such that $f(x)-f(y)> (1-\varepsilon)d(x,y)$ there exist $u,v\in M$ such that $f(u)-f(v)> (1-\varepsilon)d(u,v)$ and $d(u,v)<\frac{\varepsilon}{(1-\varepsilon)^2}d(x,y)$.
\end{lemma}

\begin{proof}
Let us consider the following functions:
\[ f_1 = f, f_2(t)=d(y,t), f_3(t)=-d(x,t), f_4(t)=f_{xy}(t) \]
We have $f_1(x)-f_1(y)> (1-\varepsilon)d(x,y)$ and $f_i(x)-f_i(y)=d(x,y)$ for $i=2,3,4$. Moreover, clearly $\norm{f_i}_{L}=1$ for $i=1,2,3$, and $\norm{f_4}_{L}=1$ as a consequence of Lemma~\ref{lemma:IKWfunction}. Consider the function $g=\frac{1}{4}\sum_{i=1}^4 f_i$. First notice that 
\[ 1\geq \norm{g}_{L}\geq \frac{1}{4} \sum_{i=1}^4 \frac{f_i(x)-f_i(y)}{d(x,y)} > 1-\frac{\varepsilon}{4}. \]
Now, the characterization given in Proposition~\ref{t:circular} provides $u,v$ in $M$ such that 
\begin{equation}\label{eq:e}(1-\varepsilon)(d(x,y)+d(u,v))\leq \min\{d(x,u)+d(y,v),d(x,v)+d(y,u)\}
\end{equation}
and $g(u)-g(v) > (1-\frac{\varepsilon}{4})d(u,v)$, that is, 
\[ \frac{1}{4}\sum_{i=1}^4 (f_i(u)-f_i(v)) > \left(1-\frac{\varepsilon}{4}\right )d(u,v) \]
Notice that each of these summands is less or equal than $d(u,v)$. Thus, we get
\begin{equation*}\label{eq:3.4}
\min\{f_i(u)-f_i(v):i\in\{1,2,3,4\}\}>(1-\varepsilon)d(u,v)
\end{equation*}
The case $i=1$ gives us $f(u)-f(v)> (1-\varepsilon)d(u,v)$. Moreover, the cases $i=2,3$ yield 
\begin{equation}\label{eq:b}\min\{d(y,u)-d(y,v),d(x,v)-d(x,u)\}>(1-\varepsilon)d(u,v).
\end{equation}
By Lemma~\ref{lemma:IKWfunction} and the case $i=4$ we have 
\begin{equation}\label{eq:d}
(1-\varepsilon)\max\{d(x,v)+d(y,v),d(x,u)+d(y,u)\}< d(x,y).
\end{equation}
The above inequalities yield
\begin{align*}
\frac{d(x,y)}{1-\varepsilon}&\mathop{>}\limits^{\mbox{(\ref{eq:d})}} d(x,u)+d(y,u)\\ 
&\mathop{>}\limits^{\mbox{(\ref{eq:b})}} d(x,u)+d(y,v)+(1-\varepsilon)d(u,v)\\
&\mathop{\geq}\limits^{\mbox{(\ref{eq:e})}} (1-\varepsilon)(d(x,y)+d(u,v))+(1-\varepsilon)d(u,v)
\end{align*}
and so 
\[2(1-\varepsilon)d(u,v)< \left(\frac{1}{1-\varepsilon}-(1-\varepsilon)\right)d(x,y) = \frac{\varepsilon(2-\varepsilon)}{1-\varepsilon}d(x,y) < \frac{2\varepsilon}{1-\varepsilon}d(x,y) \] %% Tony 12/05: no hay coma
as desired.
\end{proof}

\begin{proof}[Proof of Theorem~\ref{caracolocal}]
\ref{caracolocal1}$\Rightarrow$\ref{caracolocal2} was proved in~\cite[Theorem 3.1]{ikw}, but let us include a sketch of the proof for completeness. So assume that $M$ is a length space. Then by Proposition~\ref{prop:charlengthspace} $M$ is spreadingly local. In order to prove that $\Lip(M)$ has the Daugavet property we will apply Theorem~\ref{caragendauga} (\ref{caragendauga3}), so we will prove that, for each $f,g\in S_{\Lip(M)}$ and every $\varepsilon>0$ we have that
\[g\in \clco\left\{u\in (1+\varepsilon)B_{\Lip(M)}: \Vert f+u\Vert>2-\varepsilon\right\}.\]
Fix $n\in\mathbb N$. Since $M$ is spreadingly local we can find $r>0$ and $\delta_0>0$ such that, for every $0<\delta<\delta_0$, there are $x_1,y_1,\ldots, x_n,y_n\in M$ such that $d(x_i,y_i)<\delta$, $\frac{f(x_i)-f(y_i)}{d(x_i,y_i)}>1-\varepsilon$ holds for each $i$ and such that $B(x_i,r)\cap B(x_j,r)=\emptyset$ for all $i\neq j$. Now, for every $i\in\{1,\ldots, n\}$ and for $\delta$ small enough, we can define a $(1+\varepsilon)$-Lipschitz function $f_i:M\longrightarrow \mathbb R$ such that $f_i=f$ in $\{x_i,y_i\}$ and $f_i=g$ in $M\setminus B(x_i,r)$. Since $f_i(x_i)-f_i(y_i)=f(x_i)-f(y_i)$ for every $i$ we deduce that
\[f_i\in \left\{u\in (1+\varepsilon)B_{\Lip(M)}: \Vert f+u\Vert>2-\varepsilon\right\}\]
holds for every $i\in\{1,\ldots, n\}$. On the other hand notice that, given $x\in M$, the set $\{i\in \{1,\ldots, n\}: f_i(x)\neq g(x)\}$ is, at most, a singleton. From the definition of the Lipschitz norm we deduce that
\[\left\Vert g-\frac{1}{n}\sum_{i=1}^n f_i\right\Vert_L\leq \frac{4+2\varepsilon}{n}.\]
Since $n$ was arbitrary we can conclude that
\[ g\in \clco\left(\left\{u\in (1+\varepsilon)B_{\Lip(M)}: \Vert f+u\Vert>2-\varepsilon\right\}\right)\] 
as desired.

\ref{caracolocal2}$\Rightarrow$\ref{caracolocal3} follows since the Daugavet property passes to preduals. 

\ref{caracolocal3}$\Rightarrow$\ref{caracolocal1}. Assume that $\mathcal F(M)$ has the Daugavet property and let us prove that $M$ is a length space. 
By Proposition~\ref{prop:charlengthspace} it is enough to show that $M$ is local.

To this end, let $0<\varepsilon<\frac{1}{4}$ and $f\in S_{\Lip(M)}$ be given. Pick $x\neq y\in M$ such that $\frac{f(x)-f(y)}{d(x,y)}>1-\varepsilon$. From Lemma~\ref{lemalocal} we can find $x_1\neq y_1\in M$ such that $\frac{f(x_1)-f(y_1)}{d(x_1,y_1)}>1-\varepsilon$ and that $d(x_1,y_1)<\frac{\varepsilon}{(1-\varepsilon)^2}d(x,y)$. A new application of Lemma~\ref{lemanormante} yields the existence of $x_2\neq y_2\in M$ such that $\frac{f(x_2)-f(y_2)}{d(x_2,y_2)}>1-\varepsilon$ and that
\[d(x_2,y_2)\leq \frac{\varepsilon}{(1-\varepsilon)^2} d(x_1,y_1)<\left(\frac{\varepsilon}{(1-\varepsilon)^2}\right)^2 d(x,y).\]
Continuing in this fashion we get a pair of sequences $\{x_n\}, \{y_n\}$ in $M$ such that $\frac{f(x_n)-f(y_n)}{d(x_n,y_n)}>1-\varepsilon$ and that
\[d(x_n,y_n)<\left(\frac{\varepsilon}{(1-\varepsilon)^2}\right)^n d(x,y)\]
holds for each $n\in\mathbb N$. Thus $M$ is local as desired. 
\end{proof}

\begin{remark} According to \cite[Definition III.1.1]{hww}, a Banach space $X$ is said to be \emph{$L$-embedded} if $X^{**}=X\oplus_1 Z$ for some Banach space $Z\subseteq X^{**}$. In~\cite[Theorem 3.4]{rueda} it is proved that a separable $L$-embedded space $X$ enjoys the Daugavet property if, and only if, so does its topological dual $X^*$.

\iffalse, given a separable $L$-embedded Banach space $X$, then it enjoys the Daugavet property if, and only if, so does its topological dual $X^*$. 
\fi

Theorem~\ref{caracolocal} says that free spaces also behave this way. However, notice that $\mathcal F(M)$ is not in general an $L$-embedded space. Indeed, it follows from~\cite{GK} that for example $\mathcal F(c_0)$ is not even complemented in its bidual.
\end{remark}

\begin{remark}
The proof of~\ref{caracolocal1}$\Rightarrow$\ref{caracolocal2} in Theorem~\ref{caracolocal} actually shows that $\Lip(M)$ satisfies a stronger version of the Daugavet property whenever $M$ is a complete length space. Let us introduce some notation, coming from~\cite{BKSW05}. Given $A\subset X$, we denote by $\conv_n(A)$ the set of all convex combinations of $n$ elements of $A$. Given $x\in S_X$ and $\varepsilon>0$, we denote 
\[
 l^+(x,\varepsilon) = \{y\in (1+\varepsilon)B_X : ||x+y||>2-\varepsilon\}.
\]
The space $X$ is said to have the \emph{uniform Daugavet property} if 
\[\lim_{n\to\infty} \sup_{x,y\in S_X} d(y, \conv_n(l^+(x,\varepsilon)) = 0\]
for every $\varepsilon>0$. In~\cite{BKSW05} is proved that $X$ has the uniform Daugavet property if and only if the ultrapower $X_{\mathcal U}$ has Daugavet property for every free ultrafilter $\mathcal U$ on $\mathbb N$. They also showed that $C(K)$ with $K$ perfect and $L_1[0,1]$ have the uniform Daugavet property. Moreover, Becerra and Martin proved in~\cite{BM06} that the Daugavet and the uniform Daugavet properties are equivalent for Lindenstrauss spaces. That is also the case for spaces of Lipschitz functions. Indeed, the proof of~\ref{caracolocal1}$\Rightarrow$\ref{caracolocal2} in Theorem~\ref{caracolocal} yields that, given $f,g\in S_{\Lip(M)}$, $n\in \mathbb N$ and $\varepsilon>0$, we have
\[ d(g, \conv_n(l^+(f,\varepsilon)) \leq \frac{4+2\varepsilon}{n} \]
which goes to $0$ as $n\to \infty$. As a consequence, we get that  $\Lip(M)$ has the Daugavet property if and only if the ultrapower $\Lip(M)_{\mathcal U}$ has the Daugavet property for every free ultrafilter $\mathcal U$ on $\mathbb N$. 
\end{remark}

\section{Extremal structure of the free spaces with Daugavet property}\label{sec:exposed}

Recall that, given a Banach space $X$, a point $x\in S_X$ is said to be a \emph{strongly exposed point} of $B_X$ if there is $f\in S_{X^*}$ such that every sequence $\{x_n\}$ in $B_X$ with $\lim_n f(x_n)=f(x)$ is norm convergent to $x$. Equivalently, the slices of $B_X$ given by $f$ form a neighbourhood basis for $x$ in $B_X$ in the norm topology. In such a case we say that the functional $f$ \emph{strongly exposes} the point $x$. The set of all strongly exposed points of $\closedball{X}$ will be denoted $\strexp{\closedball{X}}$.

In what follows we will first characterise the strongly exposed points of $\closedball{\Free(M)}$ which will allow us to characterise the metric spaces $M$ such that the unit ball of the free space $\Free(M)$ has a strongly exposed point.
In a general Banach space $X$ the property that $\strexp{\closedball{X}}\neq \emptyset$ is extremely opposite to the Daugavet property. 
Our results below yield in particular that for example in the class of free spaces of compact metric spaces these properties are plainly complementary.

For starters, let us reduce the set of possible candidates for a strongly extreme point in $\closedball{\Free(M)}$.
\begin{lemma}
Let $M$ be a pointed metric space, then 
\[
\strexp{\closedball{\Free(M)}} \subset \set{\frac{\delta_x-\delta_y}{d(x,y)}:x\neq y \in M}.
\]
\end{lemma}
\begin{proof}
Assume that $\mu\in \strexp{B_{\Free(M)}}$.  
The slices of $\closedball{\Free(M)}$ determined by the strongly exposing functional form a neighborhood basis of $\mu$ in $\closedball{\Free(M)}$ equipped with the norm topology.
By~\cite[Proposition 9.1]{GMZ14}, this condition implies that $\mu$ is a \emph{preserved extreme point}, i.e. $\mu \in \ext{B_{\Lip(M)^*}}$.
Hence~\cite[Corollary 2.5.4]{wea} yields that $\mu=\frac{\delta_x-\delta_y}{d(x,y)}$ for some $x,y\in M$ with $x\neq y$.
\end{proof}

Let us introduce a bit of notation which will play a central role in the sequel.

\begin{definition}
Let $x\neq y \in M$. 
A function $f\in \justLip(M)$ is \emph{peaking at $(x,y)$} if $\frac{f(x)-f(y)}{d(x,y)}=1$ and for every open set $U$ of $M^2\setminus\set{(x,x):x\in M}$ containing $(x,y)$ and $(y,x)$, there exists $\delta >0$ such that the condition $(z,t)\notin U$ implies $\frac{|f(z)-f(t)|}{d(z,t)}\leq 1-\delta.$
\end{definition} 
This definition is equivalent to: $\frac{f(x)-f(y)}{d(x,y)}=1$ and if $\{u_n\}, \{v_n\}\subset M$, then 
\[
\lim\limits_{n\rightarrow +\infty}\frac{f(u_n)-f(v_n)}{d(u_n,v_n)}=1\Rightarrow \lim\limits_{n\rightarrow +\infty} u_n=x \textrm{\  and }	\lim\limits_{n\rightarrow +\infty} v_n=y.\]
We will say that $(x,y) \in M^2$ is a \emph{peak couple} if there is a function peaking at $(x,y)$.
 
Moreover in~\cite[Proposition 2.4.2]{wea} it is proved that if a pair of points $(x,y)$ is a peak couple then $\frac{\delta_x-\delta_y}{d(x,y)}$ is a \emph{preserved extreme point}, that is, an extreme point of $B_{\Lip(M)^*}$. Below we will give an alternative proof of this fact, showing first that every peak couple corresponds to a strongly exposed point of $B_{\mathcal F(M)}$. 

In~\cite[Proposition 2]{dkp} a characterization of peak couples $(x,y)\in M^2$ is given when $M$ is a subset of an $\Real$-tree.
We generalise this characterisation to an arbitrary metric space $M$.
We shall need the following classical notation.
Given $x,y,z \in M$ the \emph{Gromov product of $x$ and $y$ at $z$} is defined as
 \[
  (x,y)_z:=\frac12(d(x,z)+d(y,z)-d(x,y))\geq 0.
 \]
It corresponds to the distance of $z$ to the unique closest point $b$ on the unique geodesic between $x$ and $y$ in any $\Real$-tree into which $\set{x,y,z}$ can be isometrically embedded (such a tree, tripod really, always exists). Notice that $(x,z)_y+(y,z)_x=d(x,y)$ and that $(x,y)_z\leq d(x,z)$ which we will use without further comment.

\begin{definition}
We say that a pair $(x,y)$ of points in $M$, $x\neq y$ satisfies the property (Z) if for every $\varepsilon>0$ there is $z\in M\setminus\{x,y\}$ such that $(x,y)_z\leq \varepsilon\min\{d(x,z),d(y,z)\}$. 
\end{definition}
Clearly, $M$ has the property (Z) (see Definition~\ref{d:propertyZ}) if, and only if, each pair of distinct points in $M$ has the property (Z).

We are now ready to give the characterisation of strongly exposed points in $\closedball{\mathcal F(M)}$ involving all the concepts introduced above.

\begin{theorem}\label{th:charstrexp} Let $x,y\in M$, $x\neq y$. The following assertions are equivalent:
\begin{enumerate}[(i)]
\item\label{charstrexp1} $\frac{\delta_x-\delta_y}{d(x,y)}$ is a strongly exposed point of $B_{\mathcal F(M)}$.
\item\label{charstrexp2} There is $f \in \Lip(M)$ peaking at $(x,y)$, i.e. $(x,y)$ is a peak couple.
\item\label{charstrexp3} 
For every $\varepsilon>0$ 
\begin{equation}\label{e:tang2}
    \inf_{u\in M\setminus(\set{x}\cup B(y,\varepsilon))}
\frac{(y,x)_{u}}{(u,y)_x}>0 \mbox{ and } \inf_{u\in M\setminus(\set{y}\cup B(x,\varepsilon))}
\frac{(y,x)_{u}}{(u,x)_y}>0 
\end{equation}
(with the convention that  $\frac{\alpha}{0}=+\infty$).
\item\label{charstrexp4} The pair $(x,y)$ does not have the property (Z). 
\end{enumerate}
\end{theorem}

In the proof we will need the following lemma.
\begin{lemma}\label{lemma:normingsmulyan} Assume that $V\subset S_X$ is a norming subset for $X^*$. Let $v\in V$ and $f\in S_{X^*}$ be so that every sequence $\{v_n\}$ in $V$ with $\lim_n f(v_n)=f(v)$ is norm-convergent to $v$. Then $\norm{\cdot}_{X^*}$ is Fr\'echet-differentiable at $f$. Therefore, $f$ strongly exposes $v$. 
\end{lemma}
The classical Smulyan's lemma (see, e.g.~\cite[Theorem 1.4.(ii)]{DGZ93}) states that $f$ strongly exposes a point $x\in S_X$ if and only if $f$ is a point of Fr\'echet differentiability of the norm of $X^*$. The proof of Lemma~\ref{lemma:normingsmulyan} which is a slight modification of the original Smulyan's lemma is left to the reader.

\begin{proof}[Proof of Theorem \ref{th:charstrexp}]
\ref{charstrexp1}$\Rightarrow$\ref{charstrexp2} is clear.

\ref{charstrexp2}$\Rightarrow$\ref{charstrexp1}. Assume that there is $f\in S_{\Lip(M)}$ peaking at $(x,y)$. %%Tony 12/05: Aqui como queremos usar la notacion "esfera" dejo f en Lip_0
%Then the hypothesis of Lemma~\ref{lemma:normingsmulyan} holds for $V=\left\{\frac{\delta_u-\delta_v}{d(u,v)}: u\neq v\right\}$ and $v=\frac{\delta_x-\delta_y}{d(x,y)}$. Indeed, 
Assume that $\lim_{n\to\infty}\duality{f, \frac{\delta_{u_n}-\delta_{v_n}}{d(u_n,v_n)}} = 1$. Since $f$ peaks at $(x,y)$, we have $\lim_{n\to\infty} d(u_n,x) = \lim_{n\to\infty} d(v_n,y)=0$ and so
$\lim_{n\to\infty}\frac{\delta_{u_n}-\delta_{v_n}}{d(u_n,v_n)}=\frac{\delta_x-\delta_y}{d(x,y)}$.

\iffalse
Therefore
\begin{align*}
\left\lVert\frac{\delta_{u_n}-\delta_{v_n}}{d(u_n,v_n)}-\frac{\delta_x-\delta_y}{d(x,y)}\right\lVert&\leq \left\lVert\frac{\delta_{u_n}-\delta_{v_n}-(\delta_x-\delta_y)}{d(x,y)}\right\lVert\\
& \quad +\left|\frac{1}{d(u_n,v_n)}-\frac{1}{d(x,y)}\right| ||\delta_{u_n}-\delta_{v_n}||\\
&\leq \frac{d(u_n,x)+d(v_n,y)}{d(x,y)}+\left|1-\frac{d(u_n,v_n)}{d(x,y)}\right|
\end{align*}
which tends to $0$ as $n$ goes to $\infty$. 
\fi
Thus, recalling that $V=\left\{\frac{\delta_u-\delta_v}{d(u,v)}: u\neq v\in M\right\}$ is norming for $\Lip(M)$,  Lemma~\ref{lemma:normingsmulyan} yields that $\frac{\delta_x-\delta_y}{d(x,y)}$ is strongly exposed by $f$. 

\ref{charstrexp2}$\Rightarrow$\ref{charstrexp3}. Assume that there are $\varepsilon>0$ and a sequence $\{u_{n}\}\subset M \setminus (\set{x} \cup B(y,\varepsilon))$ such that 
\[
\lim\limits_{n\rightarrow +\infty}\frac{(y,x)_{u_n}}{(u_n,y)_x}=0.
\]
We then clearly have
\[
\lim\limits_{n\rightarrow +\infty}\frac{(y,x)_{u_n}}{d(x,u_n)}=0
\]
since $(u_n,y)_x\leq d(x,u_n)$.
Let $f\in \justLip(M)$ be such that $\norm{f}_L=1$ and $\frac{f(x)-f(y)}{d(x,y)}=1$. We may assume that $f(y)=0$ and $f(x)=d(x,y)$. 
Consider $b_n$ so that $\{x,y,u_n\}$ embeds isometrically into $\{x,y,u_n,b_n\}$. 
Notice that, if we denote $f_n$ the unique $1$-Lipschitz extension of $f\restricted_{\set{x,y,u_n}}$ to $\set{x,y,u_n,b_n}$, then $f_n(b_n)=(u_n,x)_y$ and therefore $\abs{(u_n,x)_y-f(u_n)}\leq (y,x)_{u_n}$. 
We have
\begin{align*}
f(x)-f(u_{n})&= (f(x)-(u_n,x)_y)-((u_n,x)_y-f(u_{n}))\\
&=(d(x,y)-(u_n,x)_y)-((u_n,x)_y-f(u_{n}))\\
&\geq (u_n,y)_x-(y,x)_{u_n}\\
&= d(x,u_n)-2(y,x)_{u_n}.
\end{align*}
It follows that 
\[
\lim\limits_{n\rightarrow+\infty}\frac{f(x)-f(u_{n})}{d(x,u_{n})}=1.
\]
and so $f$ is not peaking at $(x,y)$ as $(u_n)$ does not converge to $y$.
%% Tony 12/05: He quitado los valores absolutos en lo de arriba, para satisfacer bien la definicion de peaking.

\ref{charstrexp3}$\Rightarrow$\ref{charstrexp4}. Assume that the pair $(x,y)$ has the property (Z). Then for every $n \in \mathbb N$ there is $z_n \in M\setminus\set{x,y}$ such that $(x,y)_{z_n}\leq \frac1n \min\set{d(x,z_n),d(y,z_n)}$.
Passing to a subsequence and exchanging the roles of $x$ and $y$ we may assume that $d(x,z_n)\leq d(y,z_n)$ for all $n \in \mathbb N$.
We thus have $\displaystyle\frac{(x,y)_{z_n}}{d(x,z_n)} \to 0$ and $d(y,z_n)\geq \frac12d(x,y)$.
Therefore
\[
 \inf_{u\in M\setminus(\set{x}\cup B(y,\frac12d(y,x)))} 
\frac{(y,x)_{u}}{d(x,u)}=0. \]
Now notice that
\[ \frac{(y,x)_u}{d(x,u)}\geq \frac{(y,x)_u}{2\max\{(u,y)_x, (x,y)_u\}} = \frac{(y,x)_u}{(u,y)_x}\]
whenever the term on the left-hand side is less than $\frac{1}{2}$. It follows that
\[\inf_{u\in M\setminus(\set{x}\cup B(y,\frac12d(y,x)))} 
\frac{(y,x)_{u}}{(y,u)_x}=0,
\] 
a contradiction.

\ref{charstrexp4}$\Rightarrow$\ref{charstrexp2}. By hypothesis, there is $\varepsilon_0>0$ such that
\[ d(x,z)+d(z,y)>d(x,y)+\varepsilon_0 \min\{d(x,z),d(z,y)\} \]
for every $z\in M\setminus\{x,y\}$. 
We will show that $(x,y)$ is a peak couple. 
%% Tony 12/05: Quito lo de abajo porque solo queremos probar que (x,y) es peak couple.
%and so $\frac{\delta_x-\delta_y}{d(x,y)}$ is strongly exposed in $B_{\mathcal F(M)}$. 
To this end, fix $\varepsilon_1>0$ with $\frac{\varepsilon_1}{1-\varepsilon_1}<\frac{\varepsilon_0}{4}$ and let $f$ be the Lipschitz function defined in~\cite[Proposition 2.8]{ikw}, namely
\[ f(z):=\begin{cases}
\max\left\{\frac{d(x,y)}{2}-(1-\varepsilon_1)d(z,x), 0\right\} & \text{if } d(z,y)\geq d(z,x),\\
& \quad d(z,y)+(1-2\varepsilon_1)d(z,x)\geq d(x,y) \\
-\max\left\{\frac{d(x,y)}{2}-(1-\varepsilon_1)d(z,y), 0\right\} & \text{if } d(z,x)\geq d(z,y),\\
& \quad d(z,x)+(1-2\varepsilon_1)d(z,y)\geq d(x,y)
\end{cases}
\]
which is well defined and satisfies $\norm{f}_{L}=1$, $f(x)-f(y)=d(x,y)$, and 
\[ \frac{f(u)-f(v)}{d(u,v)}>1-\varepsilon_1 \text{ implies } \max\{d(x,u),d(y,v)\}<\frac{d(x,y)}{4} \]
for any $u,v\in M$, $u\neq v$ (see the proof of Proposition 2.8 in~\cite{ikw}). Now, take $g=\frac{1}{2}(f+f_{xy})$. We claim that $g$ peaks at $(x,y)$. Indeed, take sequences $\{u_n\}$ and $\{v_n\}$ in $M$ with $\lim_{n\to\infty} \frac{g(u_n)-g(v_n)}{d(u_n,v_n)}=1$. Fix $\varepsilon>0$ and take $0<\gamma<\varepsilon_1$ such that $\frac{\gamma}{1-\gamma}d(x,y)<\varepsilon_0\varepsilon$. Now, take $n_0$ such that 
\begin{equation}\label{eq:notzg} \frac{g(u_n)-g(v_n)}{d(u_n,v_n)}>1-\frac{\gamma}{4}
\end{equation}
for every $n\geq n_0$. We will show that $d(x, u_n), d(y, v_n)<\varepsilon$. First, note that~(\ref{eq:notzg}) implies that 
\[ \frac{ f(u_n) - f(v_n)}{d(u_n,v_n)} > 1-\frac{\gamma}{2}>1-\varepsilon_1 \]
and so $d(x,u_n), d(y,v_n)<\frac{d(x,y)}{4}$. Therefore $d(x,u_n)<d(y,u_n)$ and $d(y,v_n)<d(x,v_n)$. 
Moreover, it also follows from~(\ref{eq:notzg}) that
\[ \frac{f_{xy}(u_n)-f_{xy}(v_n)}{d(u_n,v_n)} > 1-\frac{\gamma}{2}>1-\gamma\] 
and so using Lemma~\ref{lemma:IKWfunction} we get $(1-\gamma)\max\{d(x,u_n)+d(y,u_n), d(x,v_n)+d(y,v_n)\}\leq d(x,y)$. This and the hypothesis imposed on the pair $(x,y)$ yield
\[
d(x,y)+\varepsilon_0 d(x, u_n) < d(x,u_n)+d(u_n,y) \leq \frac{1}{1-\gamma} d(x,y).\]
Therefore,
\[ d(x,u_n)\leq \frac{1}{\varepsilon_0}\left(\frac{1}{1-\gamma}-1\right)d(x,y) < \varepsilon \] 
for every $n\geq n_0$. Similarly, $d(y, v_n)<\varepsilon$. This shows that $\{u_n\}$ converges to $x$ and $\{v_n\}$ converges to $y$. Thus, $g$ peaks at $(x,y)$ as desired. %% no comma
\end{proof}

Note that Theorem~\ref{th:charstrexp} generalises~\cite[Proposition 2]{dkp}, where the equivalence between~\ref{charstrexp2} and~\ref{charstrexp3} is proved under the assumption that $M$ is a subset of an $\mathbb R$-tree.

Note that the proof of \ref{charstrexp2}$\Rightarrow$\ref{charstrexp1} in Theorem~\ref{th:charstrexp} actually shows that the following holds:

\begin{corollary} Let $M$ be a pointed metric space, $f\in\Lip(M)$ and $x,y\in M$, $x\neq y$. Then $f$ peaks at the pair $(x,y)$ if and only if $f$ strongly exposes $\frac{\delta_x-\delta_y}{d(x,y)}$ in $B_{\mathcal F(M)}$.
\end{corollary}

In what follows we show that free spaces naturally strengthen their extremal structure. Recall that, given a Banach space $X$, a point $x\in S_X$ is said to be a \emph{weakly exposed point} of $B_X$ if there is an $f\in S_{X^*}$ such that every sequence $\{x_n\}$ in $B_X$ with $\lim_n f(x_n)=f(x)$ is weakly-convergent to $x$. Note that in that case the slices of $B_X$ given by $f$ are neighbourhood basis for $x$ in the weak topology of $B_X$. Thus, every weakly exposed point is also a preserved extreme point.  

\begin{proposition}\label{prop:wese} Let $\mu$ be weakly exposed in $B_{\Free(M)}$ by $f\in S_{\Lip(M)}$. Then $\mu$ is strongly exposed by $f$.
\end{proposition}

\begin{proof} 
First note that $\mu$ is a preserved extreme point of $B_{\mathcal F(M)}$ and so $\mu = \frac{\delta_x-\delta_y}{d(x,y)}$ for some $x,y\in M$. Now take sequences $\{u_n\}, \{v_n\}$ in $M$ such that $\frac{f(u_n)-f(v_n)}{d(u_n,v_n)}\to 1$. Since $f$ weakly exposes $\mu$ we have that $\frac{\delta_{u_n}-\delta_{v_n}}{d(u_n,v_n)}\stackrel{w}{\to}\mu$. Now, a result by Albiac and Kalton \cite[Lemma 5.1]{AK} ensures that $\{\frac{\delta_{u_n}-\delta_{v_n}}{d(u_n,v_n)}\}$ is norm-convergent to $\mu$. Thus $f$ peaks at $\mu$ and so $\mu$ is strongly exposed by $f$ by Theorem \ref{th:charstrexp}. 
\iffalse
Assume that $\mu$ is not a strongly exposed point in order to get a contradiction. 
By Theorem~\ref{th:charstrexp} and exchanging the roles of $x$ and $y$ if needed, we get that there is $\varepsilon>0$ and a sequence $\{u_n\}\subset M\setminus (\{x\}\cup B(y,\varepsilon))$ such that $\lim_{n\to\infty} \frac{(y,x)_{u_n}}{(u_n,y)_x}=0$. 
Now, the argument given in~\ref{charstrexp2}$\Rightarrow$\ref{charstrexp3} of the proof of Theorem~\ref{th:charstrexp} gives that $\lim_{n\to\infty} \frac{ f(x)-f(u_n)}{d(x,u_n)}=1$ for every function $f\in S_{\Lip(M)}$ so that $\frac{f(x)-f(y)}{d(x,y)}=1$. In particular, this holds for the function which weakly exposes $\frac{\delta_x-\delta_y}{d(x,y)}$. Therefore, $\frac{\delta_x-\delta_{u_n}}{d(x,u_n)}\stackrel{w}{\to}\frac{\delta_x-\delta_y}{d(x,y)}$. Now take $g$ a Lipschitz function with $g(y)=0$ and $g\restricted_{M\setminus B(y,\varepsilon)}=d(x,y)$. We have that $\frac{g(x)-g(u_n)}{d(x,u_n)} = 0$, whereas $\frac{g(x)-g(y)}{d(x,y)}=1$, a contradiction. 
\fi
\end{proof}

As a consequence of Proposition \ref{prop:wese} we get the following:

%% LC 21/07 Creo que tenemos que añadir una hipótesis. El problema es que en otro caso no tenemos garantizado que la derivada Gateuax sea un elemento de F(M)
\begin{corollary} \label{cor:gatfre} Let $M$ be a pointed metric space and $f\in S_{\Lip(M)}$. 
If the norm of $\Lip(M)$ is G\^ateaux differentiable at $f$, then it is also Fr\'echet differentiable at $f$ (with  the derivative $\mu\in \Free(M)$ of the form $\mu=\frac{\delta_x-\delta_y}{d(x,y)}$). 
\end{corollary}

\begin{proof}
Let us show that if $f \in \Lip(M)$ does not attain its norm on $B_{\Free(M)}$, then $f$ is not a point of G\^ateaux differentiability of the norm $\norm{\cdot}_L$.
Indeed, let $\{x_n\},\{y_n\} \subset M$ be such that $\duality{f,m_{x_ny_n}} \to \norm{f}_L$. 
It is enough to show that the functional $g \mapsto \lim \duality{g,m_{x_ny_n}}$ defined on the linear span of $f$ admits two different extensions on $\Lip(M)$.

First we claim that there is a $g \in \Lip(M)$ such that $\lim \duality{g,m_{x_ny_n}}$ does not exist. 
Indeed, assume that for every $g \in \Lip(M)$ the limit exists and denote it by $\varphi(g)$.
Then $\varphi \in \Lip(M)^*$ by the uniform boundedness principle and $\norm{\varphi}\leq 1$.
Now for any two increasing sequences $\{n_k\}$ and $\{m_k\}$ of positive integers we have that $m_{x_{n_k}y_{n_k}}-m_{x_{m_k}y_{m_k}} \to 0$ weakly. 
Therefore \cite[Lemma 5.1]{AK} shows that $m_{x_{n_k}y_{n_k}}-m_{x_{m_k}y_{m_k}} \to 0$ in norm. 
So $\{m_{x_ny_n}\}$ is norm Cauchy and it follows that $\varphi \in B_{\Free(M)}$ which is a contradiction which proves our claim.

Let now $\{n_k\}$ and $\{m_k\}$ be such that $\lim \duality{g,m_{x_{n_k}y_{n_k}}}=\limsup \duality{g,m_{x_ny_n}}$ and $\lim \duality{g,m_{x_{m_k}y_{m_k}}}=\liminf \duality{g,m_{x_ny_n}}$. 
It is clear that the Hahn-Banach extensions of these limits are different and they both extend the original limit.
Thus $\norm{\cdot}_L$ is not Gateaux differentiable at the point $f$.

We now assume that the norm is G\^ateaux differentiable at $f$. 
By the previous paragraph, the unique norming functional $\mu$ belongs to $\Free(M)$.
If $\{\mu_n\}$ is a sequence in $B_{\mathcal F(M)}$ such that $\<f, \mu_n\>\to 1$ then the version of the Smulyan lemma for G\^ateux differentiability (see e.g.\cite[Theorem~1.4.(iv)]{DGZ93}) yields that $\mu_n\stackrel{w}{\to} \mu$ and so $\mu$ is  weakly exposed in $B_{\Free(M)}$ by $f$. Now apply Proposition \ref{prop:wese} and the version of Smulyan's lemma for Fr\'echet differentiability. 
\end{proof}

Finally we show that in free spaces with the Daugavet property there are no preserved extreme points.

\begin{proposition}\label{prop:amcpreserved} Let $M$ be a pointed length space. Then $B_{\mathcal F(M)}$ does not have any preserved extreme point, that is, $\ext{B_{\Lip(M)^*}}\cap \mathcal F(M) = \emptyset$.\end{proposition}

\begin{proof}
Assume that there is some preserved extreme point of $B_{\mathcal F(M)}$, which must be of the form $\frac{\delta_x-\delta_y}{d(x,y)}$ for some $x,y\in M$, $x\neq y$. Take a sequence $\{u_n\}\subset M$ such that $\max\{d(x,u_n),d(y,u_n)\}\leq \frac{1+1/n}{2} d(x,y)$ for every $n$, which exists since $M$ is a length space. Consider
\[ \mu_n = \frac{2}{1+1/n} \frac{\delta_x-\delta_{u_n}}{d(x,y)}, \quad \nu_n = \frac{2}{1+1/n}\frac{\delta_{u_n}-\delta_y}{d(x,y)}. \]
Then $||\mu_n||,||\nu_n||\leq 1$ and $\frac{\mu_n+\nu_n}{2} \stackrel{\norma}{\to} \frac{\delta_x-\delta_y}{d(x,y)}$. Since $\frac{\delta_x-\delta_y}{d(x,y)}$ is a preserved extreme point, this implies that $\mu_n\stackrel{w}{\to}\frac{\delta_x-\delta_y}{d(x,y)}$~\cite[Proposition 9.1]{GMZ14}. It follows that $\delta_{u_n}\stackrel{w}{\to}\frac{\delta_x+\delta_y}{2}$, which is impossible. 
\end{proof}

Note that the previous result proves that, if $M$ is compact, then $\mathcal F(M)$ has the Daugavet property if, and only if, $B_{\mathcal F(M)}$ does not have any preserved extreme point. 

\begin{remark}
While preparing the present paper we have learned that Aliaga and Guirao~\cite{AG} have proved that if $M$ is compact and $x,y$ are two distinct points in $M$, then $[x,y]=\{x,y\}$ if and only if the molecule $\frac{\delta_x-\delta_y}{d(x,y)}$ is a preserved extreme point of $B_{\mathcal F(M)}$. This solves in the affirmative the open problem mentioned on page 53 of \cite{wea}. Let us remark that the result does not hold in general when $M$ is not compact. Indeed, in \cite[Example 2.4]{ikw}, a length metric space $M$ is constructed such that $[x,y]=\{x,y\}$ for \emph{all} $x\neq y \in M$.
Despite this, $\closedball{\Free(M)}$ has \emph{no} preserved extreme point as is implied by Proposition~\ref{prop:amcpreserved}.
\end{remark}

Let us end the section by giving the following characterisation under compactness assumptions, which improves~\cite[Theorem 3.3]{ikw}.

\begin{corollary}\label{caracompa}
Let $M$ be a pointed compact metric space. The following assertions are equivalent:
\begin{enumerate}[(i)]
\item $M$ is geodesic.
\item For every $x, y\in M$ there is $z\in [x,y]\setminus\{x,y\}$.
\item $\Lip(M)$ has the Daugavet property.
\item The unit ball of $\mathcal F(M)$ does not have any preserved extreme point. 
\item The unit ball of $\mathcal F(M)$ does not have any strongly exposed point.
\item The norm of $\Lip(M)$ does not have any point of G\^ateux differentiability.
\item The norm of $\Lip(M)$ does not have any point of Fr\'echet differentiability. 
\end{enumerate}
\end{corollary}

\begin{proof}
The equivalence between (i) and (iii) follows from Theorem \ref{caracolocal} and the fact that compact length spaces are geodesic. Moreover, (i)$\Rightarrow$(ii) follows from Lemma \ref{lemma:charlength}. Now, if (ii) holds then every molecule $\frac{\delta_x-\delta_y}{d(x,y)}$ can be written as a non-trivial convex combination as
\[\frac{\delta_x-\delta_y}{d(x,y)} = \frac{d(x,z)}{d(x,y)}\frac{\delta_x-\delta_z}{d(x,z)} + \frac{d(z,y)}{d(x,y)} \frac{\delta_z-\delta_y}{d(z,y)}\]
and so it is not an extreme point of $B_{\mathcal F(M)}$. Since all the preserved extreme points are molecules, (iv) holds.

It is clear that (iv) implies (v). If (v) holds then by Theorem \ref{th:charstrexp} we have that $M$ has property (Z). Since $M$ is compact then Proposition 2.8 in \cite{ikw} says that $M$ is local, and so a length space by Proposition \ref{prop:charlengthspace}. This shows that (v) implies (i). Finally, the equivalence between (v), (vi) and (vii) follows from Corollary \ref{cor:gatfre} and Smulyan's lemma (and holds even in the non-compact case).
\end{proof}

\begin{remark}
Note that the previous corollary means that, whenever $M$ is a pointed compact metric space, then  either $\mathcal{F}(M)$ has the Daugavet property or its unit ball is dentable. Such extreme behaviour related to the diameter of the slices of the unit ball does not hold for its dual $\Lip(M)$. Indeed, in \cite{ivakhno} it is proved that every slice of $B_{\Lip(M)}$ has diameter two whenever $M$ is unbounded or it is not uniformly discrete. 
Consequently $M=[0,1]\cup [2,3]$ is an example of a compact metric space such that every slice of $B_{\Lip(M)}$ has diameter two but $\Lip(M)$ fails the Daugavet property.
%% Tony 12/05: No creo que hay que hacer referencia...
%Consequently, \cite[Corollary 3.4]{ikw} gives examples of compact metric spaces $M$ such that every slice of $B_{\Lip(M)}$ has diameter two but $\Lip(M)$ fails the Daugavet property.
\end{remark}

\section{Remarks and open questions}\label{sec:remarks}

Corollary \ref{caracompa} motivates the following question.

\begin{question}
Let $M$ be a metric space. If $M$ has the property (Z), is $M$ a length space?
\end{question}

Corollary \ref{caracompa} says that the answer is affirmative when $M$ is compact. 
Moreover, the affirmative answer to this problem would imply the following dichotomy for every metric space $M$: 
either $\mathcal F(M)$ has the Daugavet property or its unit ball has a strongly exposed point and, in particular, is dentable.

Though we do not know the answer to the previous question in the general case, we can give an affirmative answer in the context of subsets of an $\mathbb R$-tree.

\begin{proposition}\label{(Z)implengthRtree}
Let $T$ be a 
%pointed 
real-tree. 
Let $M \subset T$ be complete.  
If $M$ has (Z), then $M$ is geodesic.
\end{proposition}

In order to prove the previous proposition we need the following result.

\begin{proposition}
Let $M$ be a complete 
%pointed 
metric space with property (Z). Then $M$ is connected.
\end{proposition}
\begin{proof}
Let us assume that $U,V \subset M$ are clopen, disjoint and $U \cup V = M$.
Then $U \times V$ is a closed subset of the complete metric space $(M^2,d_1)$ where $d_1((a,b),(c,d))=d(a,c)+d(b,d)$.
Let $\alpha \in (0,1)$. 
By the Ekeland's variational principle applied to the function
$d\restricted_{U\times V}$ there is $(x,y) \in U\times V$ such that for every $(u,v) \in U\times V$ we have
\[
 d(x,y) \leq d(u,v)+\alpha(d(x,u)+d(y,v)).
\]
Now let $\varepsilon \in (0,1-\alpha)$ %be such that $\frac{\varepsilon}{2-\varepsilon}\frac{1}{1-\alpha}<\frac12$ %% LC Parece que no hace falta
and let $z \in M\setminus\set{x,y}$ satisfy (Z) with this $\varepsilon$.
We assume that $z \in V$ and we set $(u,v)=(x,z)$ in the above inequality.
We have
\begin{equation}\label{e:crucial}
 \begin{split}
  d(x,y)& \leq d(x,z)+\alpha d(y,z)\\
  &= d(x,z)+d(y,z)-(1-\alpha)d(y,z)\\
  &\leq d(x,y) +\varepsilon \dist(z,\set{x,y}) - (1-\alpha) d(y,z)
 \end{split}
\end{equation}
This implies that
\[ d(y,z)\leq \frac{\varepsilon}{1-\alpha} d(z,\set{x,y}) < d(z, \set{x,y}) \]
which is a contradiction.
\iffalse
This implies that 
$d(y,z) \leq \frac{\varepsilon}{2-\varepsilon}\frac{d(x,y)}{1-\alpha}<\frac{d(x,y)}2$ yielding that $d(y,z)=\dist(z,\set{x,y})$.
Substituting this back to \eqref{e:crucial} we get
$(1-\alpha)\leq \varepsilon$ which is a contradiction.
\fi
\end{proof}

This proposition yields immediately that $M$ is perfect (i.e. has no isolated points) whenever $M$ is complete and has (Z).

\begin{proof}[Proof of Proposition \ref{(Z)implengthRtree}]
If $M$ is a connected complete subset of an $\Real$-tree $T$ we get that $M$ is geodesic.
%as the connected subsets of $\Real$-trees are arcwise connected and the image of any arc connecting $x,y \in M$ coincides with the geodesic connecting $x,y$.
Indeed, for any two points $x,y \in M$ let $\varphi:[0,d(x,y)] \to T$ be the unique 1-Lipschitz map such that $\varphi(0)=x$ and $\varphi(d(x,y))=y$.
We will denote $\pi:T \to \varphi([0,d(x,y)])$ the metric projection onto $\varphi([0,d(x,y)])$.
It is well known to be continuous.
If there is $t \in (0,d(x,y))$ such that $\varphi(t)\notin M$, then by completeness of $M$ there is $\varepsilon>0$ such that $\closedball{T}(\varphi(t),\varepsilon) \cap M =\emptyset$.
Since for every $u \in U=M\cap \pi^{-1}(\varphi([0,t]))$ and every  $v \in V=M \cap \pi^{-1}(\varphi((t,d(x,y)]))$ 
we have $d(u,v)=d(u,\varphi(t))+d(\varphi(t),v)$ it follows that
$\dist(U,V)\geq 2\varepsilon$ which is impossible as $M=U\cup V$.
\end{proof}

Now we will end the section with a problem about the Daugavet property in vector-valued Lipschitz functions spaces, for which we will have to introduce a bit of notation. Given a metric space $M$ and a Banach space $X$, we consider
\[ \Lip(M,X):=\left\{f\colon M\longrightarrow X: f(0)=0\mbox{ and } \sup\limits_{x\neq y\in M}\frac{\Vert f(x)-f(y)\Vert}{d(x,y)}<\infty\right\}.\]
This space is a Banach space under the norm given by the smallest Lipschitz constant. Note that the space $\Lip(M,X)$ is isometrically isomorphic to $L(\mathcal F(M),X)$, the space of bounded linear operators from $\Free(M)$ to $X$.

\begin{proposition}\label{localvector}
Let $M$ be a length space and let $X$ be a Banach space. Then, for every Lipschitz map $f\colon M\to X$ and every $\varepsilon>0$ there are $x\neq y\in M$ such that $\frac{\Vert f(x)-f(y)\Vert}{d(x,y)}>\Vert f\Vert-\varepsilon$ and that $d(x,y)<\varepsilon$.
\end{proposition}

\begin{proof}
Pick a positive $\varepsilon$, a pair of points $u\neq v\in M$ and $x^*\in S_{X^*}$ such that
\[\frac{x^*(f(u))-x^*(f(v))}{d(u,v)}>\Vert f\Vert-\varepsilon\]
holds. This means that the real Lipschitz function $x^*\circ f$ has Lipschitz norm bigger than $\Vert f\Vert-\varepsilon$. Since $M$ is local we can find $x\neq y\in M$ such that $d(x,y)<\varepsilon$ and that
\[\Vert f\Vert-\varepsilon<\frac{x^*(f(x)-f(y))}{d(x,y)}\leq \frac{\Vert f(x)-f(y)\Vert}{d(x,y)}.\]
Since $\varepsilon>0$ was arbitrary the result follows.
\end{proof}

Let $M$ be a metric space and $X$ be a Banach space. According to~\cite{blr} the pair $(M,X)$ is said to have the \emph{contraction-extension property} if given $N\subseteq M$ and a Lipschitz map $f\colon N\longrightarrow X$, there exists a Lipschitz map $F\colon M\longrightarrow X$ extending $f$ such that
\[\Vert F\Vert_{\Lip(M,X)}=\Vert f\Vert_{\Lip(N,X)}.\]

Note that, in the particular case of $M$ being a Banach space, the definition given above agrees with the one given in~\cite{beli}.

Let us give some examples of pairs which have the contraction-extension property. First of all, given a metric space $M$, the pair $(M,\mathbb R)$ has the contraction-extension property (using the infimal convolution formula of McShane-Whitney).
In addition, in~\cite[Chapter 2]{beli} we can find some examples of Banach spaces $X$ such that the pair $(X,X)$ satisfies the contraction-extension property such as Hilbert spaces and $\ell_\infty^n$. Finally, if $Y$ is a strictly convex Banach space such that there exists a Banach space $X$ with $\dim(X)\geq 2$ and verifying that the pair $(X,Y)$ has the contraction-extension property, then $Y$ is a Hilbert space~\cite[Theorem 2.11]{beli}.

Now we can generalise~\ref{caracolocal1}$\Rightarrow$\ref{caracolocal2} in Theorem~\ref{caracolocal} to the vector-valued framework.

\begin{proposition}
Let $M$ be a pointed length space and $X$ be a Banach space such that the pair $(M,X)$ has the contraction-extension property. Then $\Lip(M,X)$ has the Daugavet property.
\end{proposition}

The proof is identical to the proof of~\ref{caracolocal1}$\Rightarrow$\ref{caracolocal2} in Theorem~\ref{caracolocal} using the contraction-extension property when appropriate.

From the above proposition we get a stability result of the Daugavet property. We will denote by $X\pten Y$ the projective tensor product of Banach spaces. For a detailed treatment and applications of tensor products, we refer the reader to \cite{rya}.

\begin{corollary}\label{corovector}
Let $M$ be a pointed metric space and $X$ be a Banach space. Then:
\begin{itemize}
\item[(a)] If the pair $(M,X)$ has the contraction-extension property and $\Lip(M)$ has the Daugavet property then $\Lip(M,X)=L(\mathcal F(M),X)$ has the Daugavet property.
\item[(b)] If the pair $(M,X^*)$ has the contraction-extension property and $\mathcal F(M)$ has the Daugavet property, then $\mathcal F(M)\pten X$ has the Daugavet property.
%% Tony 12/05: this should be better explained but not necessarily for the preprint.
%% Abraham 12/05: No sé qué debería explicarse mejor, sólo usa el apartado anterior y el hecho de que la Daugavet pasa a preduales.
%% Tony 12/05: Pues eso, que el tensor proyectivo es predual de Lip(M,X).
%% Abraham 12/05: Mmm, diría que para justificar todo eso, una frase al principio diciendo que remitimos a [BLR] para ver con detalle la definición y dualidad de los espacios de funciones Lipschitz vector valuados. Otra opción es decir que es bien conocido (por ejemplo véase Ryan) que L(X,Y) es el dual de un tensor proyectivo cuando Y es un dual.
\end{itemize}
\end{corollary}

The question whether the Daugavet property is preserved by projective tensor products from both factors  was posed in~\cite{wer}. It remains, to the best of our knowledge, unsolved. 
It is known, however, that the Daugavet property can not be preserved by projective tensor products from one factor. Indeed, in~\cite[Corollary~4.3]{kkw} an example of a complex $2$-dimensional Banach space $E$ is given so that $L_\infty^\mathbb C([0,1])\pten E$ fails to have the Daugavet property (see~\cite[Remark~3.13]{llr} for real counterexamples failing to fulfil much weaker requirements than the Daugavet property). In spite of the previous fact, we get from Corollary~\ref{corovector} that, for a Hilbert space $H$, the space $\mathcal F(H)\pten H$ has the Daugavet property, a result which we find curious, if nothing else. Moreover, Corollary~\ref{corovector} motivates the following problem.

\begin{question}
Let $M$ be a pointed metric space and $X$ a Banach space. If $\Lip(M)$ has the Daugavet property, does $\Lip(M,X)$ or $\mathcal F(M)\pten X$ have the Daugavet property?
\end{question}

Note that the same problem is open if we replace the Daugavet property with the octahedrality of the norm (see~\cite[Question 2]{blr}).

\section*{Acknowledgements} 
The third author is grateful to Departamento de Matem\'aticas de la Universidad de Murcia for the excellent working conditions during his visit in February 2017.

\end{document}